\documentclass[final]{siamltex}

\usepackage{amsmath,amssymb,epsfig}

\title{Perturbation of matrices and non-negative rank with a view toward statistical models}

\author{Cristiano Bocci\thanks{Dipartimento di Scienze Matematiche e Informatiche ``R. Magari'', Universit\`a di Siena, Pian dei Mantellini 44, 53100 Siena, Italy ({\tt cristiano.bocci@unisi.it})}
        \and Enrico Carlini\thanks{Dipartimento di Matematica, Politecnico di Torino,
Corso Duca degli Abruzzi, 24, 10129 Turin, Italy ({\tt
enrico.carlini@polito.it}).} \and Fabio
Rapallo\thanks{Dipartimento di Scienze e Tecnologie Avanzate,
Universit\`a del Piemonte Orientale, Viale Teresa Michel 11, 15121
Alessandria, Italy ({\tt fabio.rapallo@mfn.unipmn.it})}}

\begin{document}

\maketitle

\begin{abstract}
In this paper we study how perturbing a matrix changes its
non-negative rank. We prove that the non-negative rank is
upper-semicontinuous and we describe some special families of
perturbations. We show how our results relate to Statistics in
terms of the study of Maximum Likelihood Estimation for mixture
models.
\end{abstract}



\pagestyle{myheadings} \thispagestyle{plain} \markboth{C. BOCCI,
E. CARLINI AND F. RAPALLO}{PERTURBATION OF MATRICES AND
NON-NEGATIVE RANK}

\section{Introduction}

The rank of a matrix gives the least number of rank one matrices,
also known as dyadic products, needed to write the matrix as a sum of dyads.
More precisely an $n\times m$ matrix $P$ such that
$\mathrm{rk}(P)=k$ can be written as
\begin{equation} \label{dyadicexp}
P=c_1 (r_1)^t+\ldots +c_k (r_k)^t \, ,
\end{equation}
where the column vectors $c_h$ and $r_h$ have the proper sizes.
Even if $P$ has non-negative entries, the vectors $c_h$ and $r_h$,
are allowed to have negative entries. If we require the vectors to
have non-negative entries, then the least number of summands is
called the non-negative rank of $P$, namely $\mathrm{rk}_+(P)$.
The non-negativity constraints make the situation more complex,
and the non-negative rank of a matrix is harder to study than the
ordinary rank, see e.g. \cite{cohen|rothblum:93}. From this
description it is clear that $\mathrm{rk}_+(P)\geq\mathrm{rk}(P)$.
Therefore, it could be impossible to decompose a rank $k$ matrix
into the sum of exactly $k$ dyadic products $c_h(r_h)^t$, where
$c_h$ and $r_h$ are non-negative vectors. The relations between
the ordinary rank and the non-negative rank have received an
increasing attention in the last years, both from a theoretical
and an applied point of view. Some recent references are
\cite{beasley|laffey:09}, \cite{dong|lin|chu:08},
\cite{lin|chu:10}, \cite{pauca|piper|plemmons:06} and
\cite{carlini|rapallo:10c}.

Computing the non-negative rank of a matrix $P$ is related to
compute a non-negative factorization of $P$. There are many
recently proposed algorithms to deal with the problem of
non-negative matrix factorization, e.g. see \cite{lee|seung:01} or
\cite{ho|vandooren:08} for an application to stochastic matrices.
However, the non-negative factorization problem is known to be
NP-hard (\cite{vavasis:09}). Roughly speaking, we can say that
there is no efficient way to compute the non-negative rank.

In this paper we study how the non-negative rank of a matrix is
affected by small perturbations of the matrix. This is of
particular interest when the matrix arises in Probability and
Statistics. In fact, when the data entries of the matrices are
determined by experimental data, small perturbations must be taken
into account.

Here, a perturbation is intended in the following topological
sense. Given a matrix $P$ we consider a neighborhood of $P$ in the
topology induced by the Frobenius norm on matrices. We call any
matrix in the neighborhood a perturbation of $P$. Clearly this
notion is more meaningful and interesting when a small
neighborhood is considered and hence matrices close to $P$ are
studied.

We show that the non-negative rank is upper-semicontinuous with
respect to the Frobenius norm, see Theorem
\ref{uppersemicontinuos}, and hence it cannot be decreased by
small perturbations of the matrix. We also produce examples of
perturbations preserving the non-negative rank, see Proposition
\ref{barycenter}. Using a Jacobian analytic approach we show that,
under some mild conditions, perturbing a matrix leaving
the ordinary rank fixed also leaves the non-negative rank
unchanged, see Proposition \ref{isorankprop}.

The notion of non-negative rank has also relevant applications in
Probability and Statistics. In fact, a probability matrix with
dyadic expansion as in Equation \eqref{dyadicexp} belongs to the
mixture of $k$ independence models for categorical data (in the
case that all the involved vectors are non-negative). Mixture
models play a central role in applied probability, as they are the
key tool in modelling partially observed phenomena, see
\cite{agresti:02} for more details. Three major topics in
Probability and Statistics where mixture models are used as a key
ingredient are: $(a)$ the study of sequence alignment, with
special attention to DNA sequences and phylogenetic trees, see
e.g. \cite{pachter|sturmfels:05,allman|rhodes:07}, and the book
\cite{shonkwiler:09} for a detailed construction of the underlying
mathematical models; $(b)$ the cluster analysis for categorical
multidimensional data, see e.g. \cite{govaert|nadif:10}; $(c)$
multivariate methods for text mining, see e.g.
\cite{zhai|velivelli|yu:04}.

There are many unsolved problems concerning mixture models. Among
these, one of the most important is the determination of the
maximum likelihood estimators. Despite the fact that Maximum
Likelihood Estimation (MLE) is a largely investigated topic, and
many numerical solutions are available, a complete theoretical
solution is not available yet. Therefore, any advance in the
geometric description of such models can be useful to address the
maximization problem from the theoretical viewpoint.

Recently, mixture models for categorical data have been considered
also in the framework of Algebraic Statistics, a branch of
Statistics which uses notions and techniques from Computational
Algebra and Algebraic Geometry, see
\cite{pachter|sturmfels:05,drton|sturmfels|sullivant:09,gibiliscoetal:10}.
In this paper, we show how the geometric description of the set of
matrices with fixed non-negative rank leads to a better
understanding of MLE for mixture models.

The paper is structured as follows. In Section \ref{basicsection}
we recall some basic notions. In Section \ref{uppersection} and
Section \ref{jacobiansection} we use a topological and analytic
approach to study perturbations. In Section \ref{examples} we use
our results to work out some significant examples. Finally, in
Section \ref{statsection} we show how our results relate to the
study of MLE in Statistics.

\section{Basic facts}\label{basicsection}

In this section, we recall some known facts about the non-negative
rank. The definitions and the results presented below will be used
throughout the paper.

\noindent{\bf Non-negative matrices.} A non-negative $n\times m$
matrix is a point in $\mathbb{R}_{\geq 0}^{nm}$ where

\[\mathbb{R}_{\geq 0}^{nm}=\left\{(p_{i,j}) : p_{i,j}\in\mathbb{R}, p_{i,j}\geq 0\right\}.\]

\noindent{\bf Stochastic matrices.} A stochastic matrix is a
non-negative matrix having column sums equal to one. To each
non-negative matrix without zero columns, we can associate a
stochastic matrix. Denote by $P=[c_1, \dots, c_m]$ the set of
columns of a non-negative matrix $P$, where $c_j\neq 0$ for all
$j$. Define the scaling factor $\sigma(P)$ by
\[
\sigma(P):=\mbox{diag} \{||c_1||_1, \dots, ||c_m||_1 \}
\]
where $||\cdot ||_1$ is the $1$-norm in $\mathbb{R}^n$. Then the pullback map $\theta$ defined by
\[
\theta (A) = A \sigma (A)^{-1}
\]
produces stochastic matrices.

{\em Remark.} In Probability, stochastic matrices are defined as the
non-negative matrices having row sums equal to one. Here we adopt
the convention of normalizing the columns. The rank and the
non-negative rank are clearly invariant under matrix transposition. Thus this
convention does not affect our results.

\noindent{\bf Simplex.} The $n$-simplex in $\mathbb{R}^{n}$ is
\[ \Delta^n=\left\{ (x_1,\ldots,x_n) \in \mathbb{R}^{n} \ :\  x_i \geq 0, \sum_{i=1}^n x_i \leq 1 \right\} \, .\]
Note that an $n\times m$ stochastic matrix $P$ can be seen as a
collection of $m$ points in $\Delta^n$. More precisely we consider
the map $\pi_n$ assigning to a matrix the set of its columns, that
is $\pi_n(P)=\{c_1,\ldots,c_m\}\subset \Delta^{n}$. All the points
$c_j$ lie on the same face of the $n$-simplex, which is a
$(n-1)$-simplex. Hence, by dropping the last component of each
$c_j$, we have a map $\pi_{n-1}$ sending $P$ into a collection of
$m$ points in $\Delta^{n-1}$. Following \cite{lin|chu:10} we will
use this geometric interpretation to visualize small size matrices
and their ranks, see Section \ref{examples}.

\noindent{\bf Non-negative rank.} Given a  $n\times m$
non-negative matrix $P$, the non-negative rank of $P$ is the
smallest integer $k$ such that
\[P=c_1 (r_1)^t+\ldots +c_k (r_k)^t\]
where the vectors $c_h\in\mathbb{R}^n$ and the vectors
$r_h\in\mathbb{R}^m$ have non-negative entries. The non-negative
rank of the matrix $P$ is denoted with $\mathrm{rk}_+(P)$. For
different description of the non-negative rank we refer the reader
to \cite{cohen|rothblum:93}.

{\em Remark.} For a non-negative matrix $P$ without zero columns,
one has that $\mathrm{rk}_+(P)=\mathrm{rk}_+(\theta(P))$. Hence,
the study of the non-negative rank of stochastic matrices
coincides with the study of the non-negative rank of non-negative
matrices without zero columns (see \cite{lin|chu:10}). Thus, from
now on, we will often restrict our attention to stochastic
matrices.

\noindent{\bf Combinatorics, Geometry and non-negative rank.}
The Nested Polytopes Problem (NPP) is introduced in
\cite{gillis|glineur:10} inspired by the intermediate simplex
problem of \cite{vavasis:09}. We can state a simplified version of
NPP directly related to the study of the non-negative rank. Let
$\mathcal{P}$ be a polytope. Given a set of $r$  points $Z$ in
$\mathcal{P}$ is it possible to find $k$ points ($k<r$) in
$\mathcal{P}$ having convex hull $\mathcal{P}_k$ such that
\[Z\subset\mathcal{P}_k\subset\mathcal{P}?\]

The following result is contained in \cite{hazewinkel:84} and it
will be of crucial importance in this paper. Thus we provide a
proof here for the convenience of the reader.

\begin{lemma}\label{nestedpolygons}
Let $P$ be an $n\times m$ stochastic matrix. If we set
$Z=\pi_{n-1}(P)$, then

\[\mbox{rk}_+(P)=\min_t\left\{ t : Z\subset\mathcal{P}_t\subset\Delta^{n-1} \right\}\]

where $\mathcal{P}_t$ is the convex hull of $t$ points.

\end{lemma}
\begin{proof}
If $\mbox{rk}_+(P)=k$, then $P=c_1 (r_1)^t+\ldots +c_k (r_k)^t$,
where the vectors $c_h$ and $r_h$ have non-negative entries. As
$P$ is a stochastic matrix, we can consider the dyads
$\frac{c_h}{||c_h||_1} (||c_h||_1 r_h)^t$. Hence, without loss of
generality, we can assume that the following hold:
\begin{itemize}

\item $||c_h||_1=1$, for all $1\leq h\leq k$;

\item $(r_h)^t=(r_h^{(1)},\ldots,r_h^{(m)})$ and $\sum_h
r_h^{(j)}=1$ for all $1\leq j\leq m$.
\end{itemize}

Let $(c_h)^t=(c_h^{(1)},\ldots,c_h^{(n)})$ and consider the points
\[Q_h=(c_h^{(1)},\ldots,c_h^{(n-1)})\in\Delta^{n-1}.\]
It is now straightforward to check that $\pi_{n-1}(P)$ is
contained in the convex hull of the points $Q_1,\ldots,Q_k$.
Hence, $\mbox{rk}_+(P)=k\geq \min_t\left\{ t :
Z\subset\mathcal{P}_t\subset\Delta^{n-1} \right\}$.

Conversely, if $k=\min_t\left\{ t :
Z\subset\mathcal{P}_t\subset\Delta^{n-1} \right\}$ then
$\pi_{n-1}(P)=\{c_1,\ldots,c_m\}$ is in the convex hull of points
$Q_1,\ldots,Q_k\in\Delta^{n-1}$. Namely
\[c_j=\sum_{h=1}^k \alpha_h^{(j)}Q_h\]
and $\sum_h\alpha_h^{(j)}=1$. If we let
$Q_h=(q_h^{(1)},\ldots,q_h^{(n-1)})$ then
\[P=\sum_{h=1}^k
\left(\begin{array}{c} q_h^{(1)} \\ \vdots \\ q_h^{(n-1)} \\
1-\sum_{i=1}^{n-1} q_h^{(i)}\end{array}\right)
\left(\begin{array}{ccc} \alpha_h^{(1)} & \ldots & \alpha_h^{(m)}
\end{array}\right).\]
Hence $\min_t\left\{ t : Z\subset\mathcal{P}_t\subset\Delta^{n-1}
\right\}=k\geq\mbox{rk}_+(P)$ and the result follows.
\end{proof}

{\em Remark.} One can also consider the NPP for
$Z\subset\mathcal{P}_k\subset\Delta^{n-1}\cap H$ where $H$ is the
linear span of $Z$. This amounts to study the restricted
non-negative rank as described in \cite{hazewinkel:84} and
\cite{gillis|glineur:10}.

\section{Upper-semicontinuity of non-negative rank}\label{uppersection}

In this section we will use the ideas recalled in Section
\ref{basicsection} to show that the non-negative rank is
upper-semicontinu\-ous in the topology given by the Frobenius
norm.

Given a non-negative matrix $P=(p_{i,j}) \in\mathbb{R}^{nm}_{\geq 0}$ and
$\epsilon>0$ define the ball of center $P$ and radius $\epsilon$

\[B(P,\epsilon)=\left\{N=(n_{i,j})\in\mathbb{R}^{nm}_{\geq 0} : \sqrt{\sum (p_{i,j}-n_{i,j})^2}<\epsilon\right\}.\]

\begin{theorem}\label{uppersemicontinuos}
Let $P$ be an $n \times m$ non-negative matrix, without zero
columns, such that $\mathrm{rk}_+(P)=k$, then there exists a ball
$B(P,\epsilon)$ such that $\mbox{rk}_+(N)\geq k$ , for all $N\in
B(P,\epsilon)$.
\end{theorem}

\begin{proof} We give a proof by contradiction. Suppose that for
all natural numbers $r$ there exists  $N(r)\in B(P,{\frac 1 r})$
such that $\mbox{rk}_+(N(r))=t<k$. Clearly, the limit of the
sequence $N(r)$ is $P$. By hypothesis we know that there exist
convex polytopes $\mathcal{P}(r)\subset\Delta^{n-1}$ such that
\[\pi_{n-1}\circ\theta(N(r))\subset\mathcal{P}(r),\]
where each $\mathcal{P}(r)$ is the convex hull of the points
\[q_1(r),\ldots, q_t(r)\in\Delta^{n-1}.\]
We now {\it claim} that there exists a limit polytope $\bar
{\mathcal{P}}\subset\Delta^{n-1}$ which is the convex hull of
points $\bar q_1,\ldots,\bar q_t$ (possibly not distinct) obtained
by the sequences $q_h(r)$. As $t<k$ it is enough to show that
$\pi_{n-1}\circ\theta(P)\subset\bar {\mathcal{P}}$ to get a
contradiction using Lemma \ref{nestedpolygons}.

Let
\[\pi_{n-1}\circ\theta(N(r))=\left\{c_1(r),\ldots,c_m(r)\right\}\]
and \[\pi_{n-1}\circ\theta(P)=\left\{c_1,\ldots,c_m\right\}\] and
notice that (possibly after reordering) the limit of $c_j(r)$ is $c_j$. Also notice that for
each $j$ we have
\[c_j(r)=\alpha_{j,1}(r)q_1(r)+\ldots+\alpha_{j,t}(r)q_t(r)\]
where the coefficients $\alpha_{j,h}(r)$ vary in the compact set
$[0,1]$, i.e. $c_j$ belongs to the convex hull of the points
$q_1(r),\ldots,q_t(r)$. Taking subsequences and passing to the
limit (limits exist as our sequences have values in compact sets)
for each $j$ we get
\[c_j=\bar \alpha_{j,1}\bar q_1+\ldots+\bar \alpha_{j,t}\bar q_t\]
and hence $c_j\in\bar{\mathcal{P}}$, for $j=1,\ldots,m$. Thus a
contradiction and the statement is proved.

{\it Proof of the claim} The sequences $q_h(r)$ have values in the
compact set $\Delta^{n-1}$ and hence they each have converging
subsequences. To show that a limit polytope exists, we proceed as
follows. Take a subsequence of $q_1(r)$ and let $\bar q_1$ be its
limit. Then, either $q_2(r)$ has a subsequence with limit $\bar
q_2\neq\bar q_1$ or it does not and, in this case, we set $\bar
q_2=\bar q_1$. In the latter case the limit polytope will be the
convex hull of strictly less than $t$ distinct points. Iterating
the process we obtain points $\bar q_1,\ldots,\bar q_t$ and their
convex hull is the limit polytope $\bar{\mathcal{P}}$.
\end{proof}

Thus, given a matrix $P$, we know that in a suitable neighborhood
of $P$ the non-negative rank can only increase, i.e. the
non-negative rank is upper-semiconti\-nuous.

Clearly each neighborhood of a matrix $P$ contains matrices having
the same non-negative rank of $P$. Consider, for example, the
matrices $\lambda P$ for $\lambda\in\mathbb{R}$ close to one. But
even more is true as shown in the following statement.

\begin{proposition}[Barycentric perturbation]\label{barycenter} Let $P$ be a non-negative
$n\times m$ matrix, without zero columns, such that
$\mathrm{rk}(P)>1$. For any $\epsilon>0$ there exists $N_\epsilon
\in B(P,\epsilon)$ such that
\[\mathrm{rk}_+(N_\epsilon)=\mathrm{rk}_+(P)\]
and $N_\epsilon\neq \lambda P$ for any $\lambda\in\mathbb{R}$.
\end{proposition}
\begin{proof} Clearly, it is enough to prove the result for small $\epsilon$. Thus, we can assume that each matrix in
$B(P,\epsilon)$ has non-negative rank at least $\mathrm{rk}_+(P)$,
i.e. $\epsilon$ is small enough for Theorem
\ref{uppersemicontinuos} to apply.

Let $P$ have columns $c_j$ and consider the vector $b={\frac 1
m}\sum_{j=1}^m c_j$. Roughly speaking we consider the barycenter of
the points $\pi_{n-1}\circ\theta(P)$. Then we consider the
$n\times m$ matrix $N_\delta$ having the $j$-th column defined as
\[c_j+{\delta}(b-c_j),\]
for $\delta\in [0,1]$. When $\delta$ moves from zero to one, the
points $\pi_{n-1}\circ\theta(N_\delta)$ approach the barycenter
$b$. Thus, by Lemma \ref{nestedpolygons}, we have that
$\mathrm{rk}_+(N_\delta)\leq\mathrm{rk}_+(P)$ for $\delta\in
[0,1]$.

By choosing $\delta$ small enough we get $N_\delta\in
B(P,\epsilon)$. Hence we have that
$\mathrm{rk}_+(N_\delta)=\mathrm{rk}_+(P)$. Letting
$N_\epsilon=N_\delta$ the existence part of the proof is done.

To complete the proof, we only have to show that $N_\epsilon$ and
$P$ are not proportional. If $N_\epsilon=\lambda P$ for some
$\lambda\in\mathbb{R}$, then, by the construction of $N_\epsilon$,
either all the columns $c_j$ are proportional or $b=0$. As
$\mathrm{rk}(P)>1$ the matrix cannot have proportional columns.
Moreover, $P$ is non-negative, thus $b$ cannot be the zero vector.
Hence, if $N_\epsilon=\lambda P$, we get a contradiction and this
completes the proof.
\end{proof}

\section{Jacobian approach}\label{jacobiansection}

Throughout this section we assume $k \leq \min \{ n,m \}$ and we
let $X_{n \times m,k}\subset \mathbb{R}^{mn}$ be the variety of
$n\times m$ matrices of rank at most $k$. It is well-known that
$\mathrm{dim}(X_{n \times m,k})=k(n+m-k)$, e.g. see
\cite[Proposition 12.2]{harris:92}.

Consider the map $f:\mathbb{R}^{k(n+m)} \to X_{n \times
m,k}\subseteq\mathbb{R}^{mn}$ which sends the point
\[
p=\left(x_{1,1}, \dots ,x_{1,n}, y_{1,1}, \dots ,y_{1,m}, \dots ,
x_{k,1}, \dots ,x_{k,n}, y_{k,1}, \dots ,y_{k,m}  \right)
\]
to the matrix
\begin{equation}\label{feqn}
f(p)= \sum_{h=1}^k
\begin{pmatrix}
x_{h,1}\\
\vdots \\
x_{h,n}\\
\end{pmatrix}
\begin{pmatrix}
y_{h,1} & \dots & y_{h,m}
\end{pmatrix}.
\end{equation}

Let $f_+$ be the restriction of $f$ to the non-negative orthant
$\mathbb{R}_{\geq 0}^{k(n+m)}$. The image of $f_+$ is the set $X^+_{n
\times m,k}$ of $n\times m$ matrices of non-negative rank at most
$k$. It is clear that $X^+_{n\times m,k}\subseteq X_{n \times m,k}$.

{\em Remark.} We let $f^*_+(p)$ be the Jacobian matrix of $f_+$ at
$p$ and we say that $f^*_+(p)$ has maximal rank if its rank is
$k(n+m-k)$.  Thus, if $f^*_+(p)$ has maximal rank, then the map
$f_+$ is locally surjective at $p$, e.g. see \cite[page 25 Corollary (d)]{warner:71}.

We can use this Jacobian approach to investigate properties of the
non-negative rank under perturbations preserving the rank.

\begin{proposition}[Isorank perturbation]\label{isorankprop}
Let $P$ be an $n\times m$ non-negative matrix such that
$\mathrm{rk}_+(P)=k$ and consider the map $f$ as defined in
\eqref{feqn}. If $P=f_+(p)$ is such that $f^*_+(p)$ has maximal
rank and $p$ has positive coordinates, then there exists a ball
$B(P,\epsilon)$ such that for each $N\in B(P,\epsilon)$ we have:
\[\mbox{if }\mathrm{rk}(N)=\mathrm{rk}(P),\mbox{then }\mathrm{rk}_+(N)=\mathrm{rk}_+(P).\]
\end{proposition}
\begin{proof}
By the hypothesis we get that $f_+$ is locally surjective at $p$.
Hence, there exist balls $B(P,\epsilon)$ and $B(p,\delta)$ such
that each $N\in B(P,\epsilon)\cap X_{n \times m,k}$ has a preimage
in $B(p,\delta)$ using the map $f_+$. Moreover, if $p$ has
positive coordinates we can find, possibly smaller, $\epsilon$ and
$\delta$ such that $B(p,\delta)$ is in the positive orthant. Thus,
given $N\in B(P,\epsilon)\cap X_{n \times m,k}$ there exists $q$
with positive coordinates such that $f(q)=N$. Hence
$\mathrm{rk}_+(N)\leq k=\mathrm{rk}_+(P)$. The conclusion follows
by Theorem \ref{uppersemicontinuos} by taking an $\epsilon$ small
enough.
\end{proof}

{\em Remark.}  The proof above also shows that there exists a
neighborhood $U$ of $P$ such that each matrix in $U$ of rank at
most $k$ has non-negative rank at most $k$. In other words, if
$P=f_+(p)$, $f^*_+(p)$ has maximal rank and $p$ has non-negative
coordinates, then $X^+_{n\times m,k}\cap U= X_{n \times m,k} \cap
U$ for $U$ a suitable neighborhood of $P$.

Now we describe sufficient conditions on $p$ granting that
$f^*_+(p)$ has maximal rank.

\begin{theorem}\label{jacobianthm} If $p\in \mathbb{R}^{k(n+m)}$ is a point with coordinates
\[p=\left(x_{1,1}, \dots ,x_{1,n}, y_{1,1}, \dots ,y_{1,m}, \dots,
x_{k,1}, \dots ,x_{k,n}, y_{k,1}, \dots ,y_{k,m} \right) ,\] such
that
\begin{itemize}
\item $(x_{h,1}, \dots, x_{h,n})$, $h=1, \dots, k$ are linearly
independent vectors of $\mathbb{R}^n$;

\item $(y_{h,1}, \dots, y_{h,m})$, $h=1, \dots, k$ are linearly
independent vectors of $\mathbb{R}^m$;

\end{itemize}

then $\mathrm{rk}(f^*_+(p))= k(n+m-k)$.
\end{theorem}

\begin{proof}
We may assume that $n\leq m$. Since the Jacobian is given by all
possible derivatives with respect to $x_{h,i}$ and $y_{h,j}$, it
is enough to show that exactly $k(m+n-k)$ of them are linearly
independent. First of all we notice that the derivative with
respect to $x_{h,i}$ is a matrix of the form
\[f_{x_{h,i}}=
\begin{pmatrix}
0\\

\vdots \\
0\\
1\\
0\\
\vdots \\
0\\
\end{pmatrix}
\begin{pmatrix}
y_{h,1} & \dots & y_{h,m}
\end{pmatrix},
\]
where the one is in position $i$. Similarly, the derivative with
respect to $y_{h,j}$ is a matrix of the form
\[f_{y_{h,j}}=
\begin{pmatrix}
x_{h,1}\\
\vdots \\
x_{h,n}\\
\end{pmatrix}
\begin{pmatrix}
0 & \dots & 0 & 1 & 0 & \dots & 0
\end{pmatrix},
\]
where the one is in position $j$. That is, the derivative with
respect to $x_{h,i}$, $i=1, \dots n$ is a matrix with all zeros
but the $i$-th row consisting of the vector $(y_{h,1}, \dots,
y_{h,m})$. Similarly the derivative with respect to $y_{h,j}$,
$j=1, \dots m$ is a matrix with all zeros but the $j$-th column
consisting of the vector $(x_{h,1}, \dots, x_{h,n})$.

We now build a set consisting of $k(m+n-k)$ linearly independent
derivatives and hence we prove the statement.

The derivatives $f_{x_{h,i}}, i=1,\ldots,n$ are clearly linearly
independent and we let $\mathcal S$ be the $kn$-dimensional vector
space that they span. Thus, if $m=k$ we are done.

If $m<k$ we proceed as follows. Let $V=\langle
(y_{h,1},\ldots,y_{h,m}),h=1,\ldots,k \rangle$ where $\dim V=k$ by
hypothesis. Now consider all the vectors in $\mathbb{R}^m$ with at
most one non-zero component and notice that they span a vector
space of dimension $m>k$. Hence, $V$ can not contain all these
vectors and we may assume that $(1,0,\ldots,0)\not\in V$. Thus it
is easy to see that the linear span
\[\mathcal{S}_1=\langle\mathcal{S},f_{y_{h,1}}\mbox{ such that }1\leq h\leq k\rangle\]
is such that $\dim \mathcal{S}_1=\dim \mathcal{S}+k=kn+k$.

If $m-k=1$ we are done. If $m-k>1$ we argue as above. Namely, as
$\dim V=k$, $V$ can not contain all vectors with first component
one and at most one more non-vanishing component. In particular,
we may assume that $(1,*,0,\ldots,0)\not\in V$, where $*$ is any
non-zero real number. Then we consider
\[\mathcal{S}_2=\langle\mathcal{S}_1,f_{y_{h,2}}\mbox{ such that }1\leq h\leq k\rangle\]
and we readily see that $\dim \mathcal{S}_2=\dim
\mathcal{S}_1+k=kn+2k$.

For each $j\leq m-k$  we can repeat the process above increasing
the dimension of $\mathcal{S}_j$ by $k$ each time.  Hence, we can
construct $\mathcal{S}_{m-k}$ such that it is spanned by
derivatives and $\dim \mathcal{S}_{m-k}=k(n+m-k)$. The statement is now
proved.
\end{proof}

\section{Examples}\label{examples}

In this section we will present some interesting examples. Some of
these examples were inspired to us by the matrix
\[
\begin{pmatrix}
1 & 0 & 1 & 0\\
1 & 0 & 0 & 1\\
0 & 1 & 1 & 0\\
0 & 1 & 0 & 1
\end{pmatrix}
\]
which is the most well-known example of a matrix with rank three
and non-negative rank four (see \cite{cohen|rothblum:93}).

\noindent{\bf Small cases.} Let $P$ be an $n\times m$ matrix and
assume $n\leq m$. We want to describe how the non-negative rank of
$P$ changes under perturbations for small values of $n$. If $n\leq
3$, then it is easy to show that
$\mathrm{rk}(P)=\mathrm{rk}_+(P)$, see \cite{cohen|rothblum:93}.
Thus the first interesting cases are for $n=4$. If
$\mathrm{rk}_+(P)=4$ then, by Theorem \ref{uppersemicontinuos},
any small perturbation will not change the non-negative rank.
Thus, let us assume that $\mathrm{rk}_+(P)=3$. Using Proposition
\ref{barycenter}, we know that there are small perturbations
preserving the non-negative rank. Of course, there are small
perturbations not preserving it: it is enough to increase the
ordinary rank. Hence, we ask: are there small perturbations of
$P$, say $P_\epsilon$, such that
$\mathrm{rk}(P_\epsilon)=\mathrm{rk}(P)$ and
$\mathrm{rk}_+(P_\epsilon)=4$? Not surprisingly, the answer
depends on the choice of $P$. It is easy to construct a matrix $P$
with the required ranks and satisfying the hypothesis of
Proposition \ref{isorankprop}. Thus, in this case, the answer to
our question is no. But, for a different choice of $P$, the answer
can be yes. Consider, for example, $P_\epsilon$ defined as
follows:
\[
P_\epsilon=\frac{1}{4}\begin{pmatrix}
2 & 0 & 2 & 1-\epsilon\\
0 & 2 & 0 & 1+\epsilon\\
0 & 0 & 2 & 1+\epsilon\\
2 & 2 & 0 & 1- \epsilon
\end{pmatrix},
\]
and let $P=P_0$. It is easy to see that
$\mathrm{rk}(P_\epsilon)=3$ for all $\epsilon$ while
$\mathrm{rk}_+(P_0)=3$ and $\mathrm{rk}_+(P_\epsilon)=4$ for small
positive values of $\epsilon$. To see this we use the graphical
presentation in Figure \ref{bcr1} where we denote with
$c_1,\ldots,c_4$ the points corresponding to the columns of $P$
while $c_4(\epsilon)$
 corresponds to the fourth column of $P_\epsilon$. In Figure \ref{bcr1}, and in
the following figures, we use the graphic representation described
in \cite{lin|chu:10} and related to the map $\pi_3$. More
precisely, a $4\times 4$ matrix will be presented as a set of four
points in a tetrahedron. This presentation allows for an easy
visualization of rank related properties. We notice, for example,
that a rank three matrix will correspond to four coplanar points.

\begin{figure}
\begin{center}
\epsfig{file=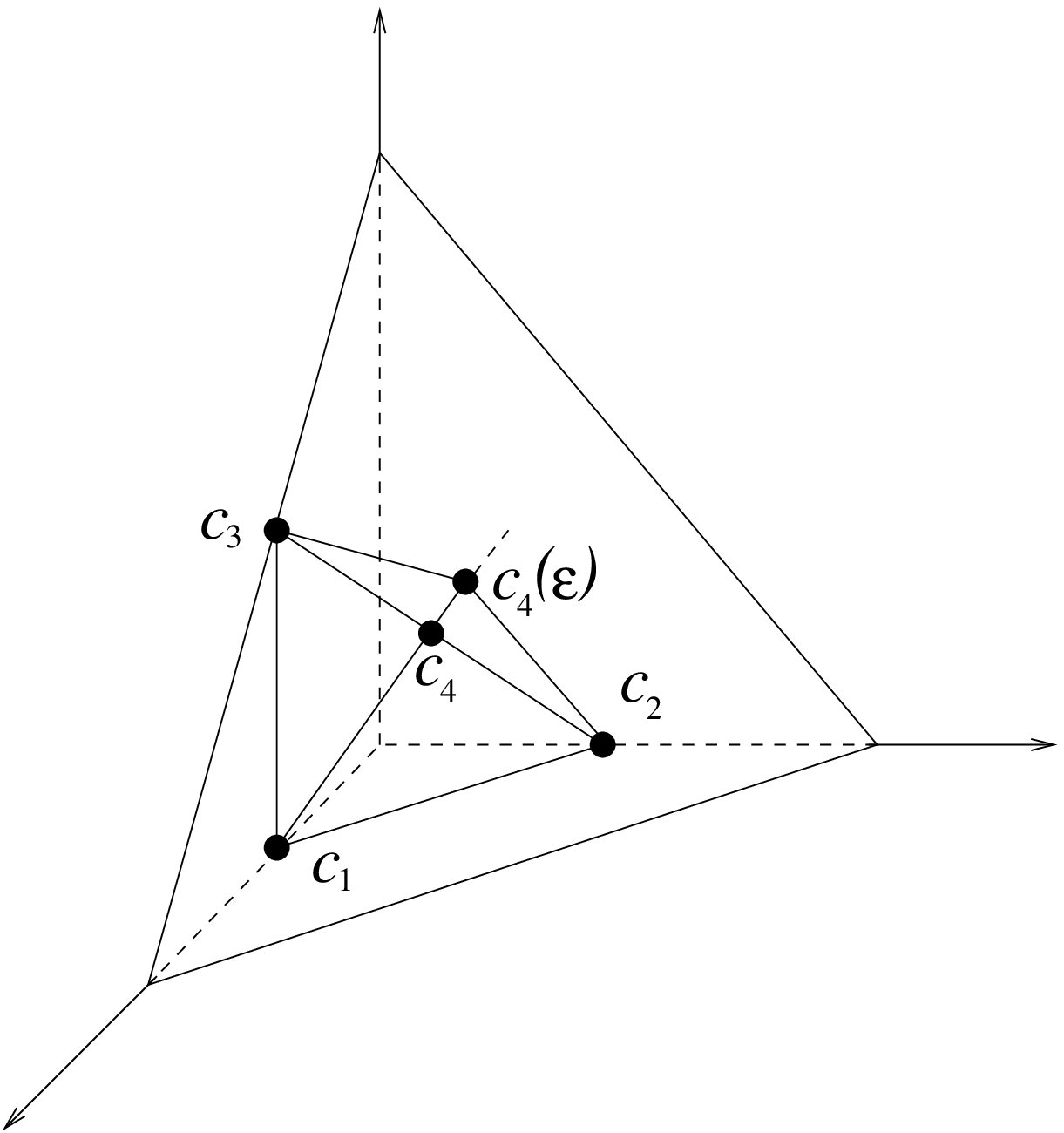, width=6cm} \epsfig{file=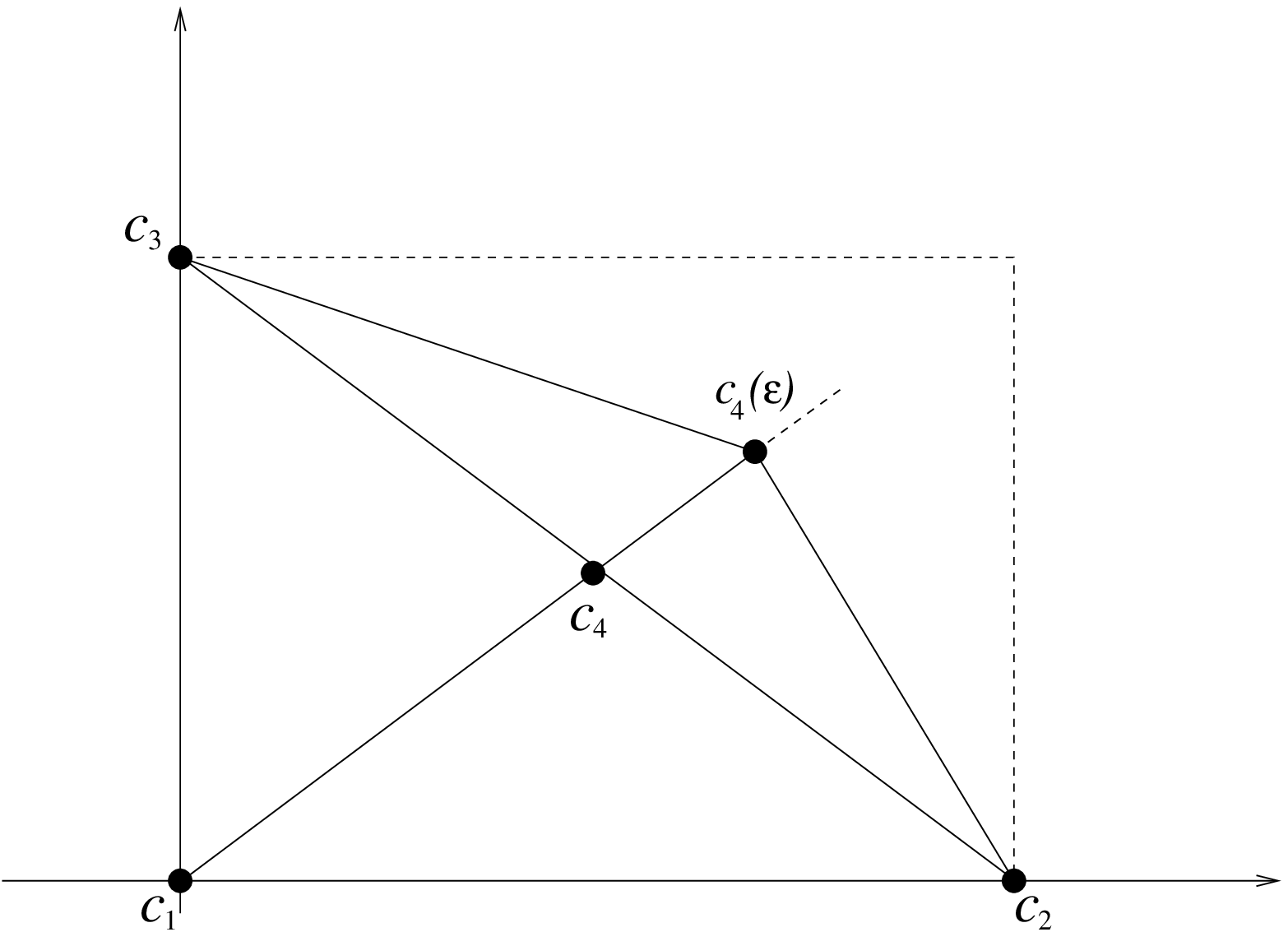, width=6cm}
\end{center}

\caption{The matrices $P_\epsilon$ for $\epsilon=0$ and a small
positive value of $\epsilon$ represented in the tetrahedron and in
the plane.}\label{bcr1}
\end{figure}

\noindent{\bf Failing of upper-semicontinuity.} The
upper-semicontinuity of the non-negative rank is of course a local
property as shown by the following example. Consider the matrix
\[
M_\epsilon=\frac{1}{2(1+2\epsilon)}\begin{pmatrix}
1+ \epsilon & \epsilon & 1+ \epsilon & \epsilon\\
1+ \epsilon & \epsilon & \epsilon & 1+ \epsilon\\
\epsilon & 1+ \epsilon & 1+ \epsilon & \epsilon\\
\epsilon & 1+ \epsilon & \epsilon & 1+ \epsilon
\end{pmatrix}
\]
and let $c_1(\epsilon),\ldots,c_4(\epsilon)$ be the four column
vectors where we set $c_j=c_j(0),i=1,\ldots,4$. When $\epsilon=0$
the matrix has non-negative rank equal to four. We use the map
$\pi_3$ to represent the columns in the simplex $\Delta^3$, which
is a tetrahedron in $\mathbb{R}^3$. To simplify the drawings, we
have dropped the first coordinate of each column instead of the
last one, but of course this does not affect our analysis. The
four points for the matrix $M_0$ are plotted in Figure
\ref{tetratrian} (left).

\begin{figure}
\begin{center}
\epsfig{file=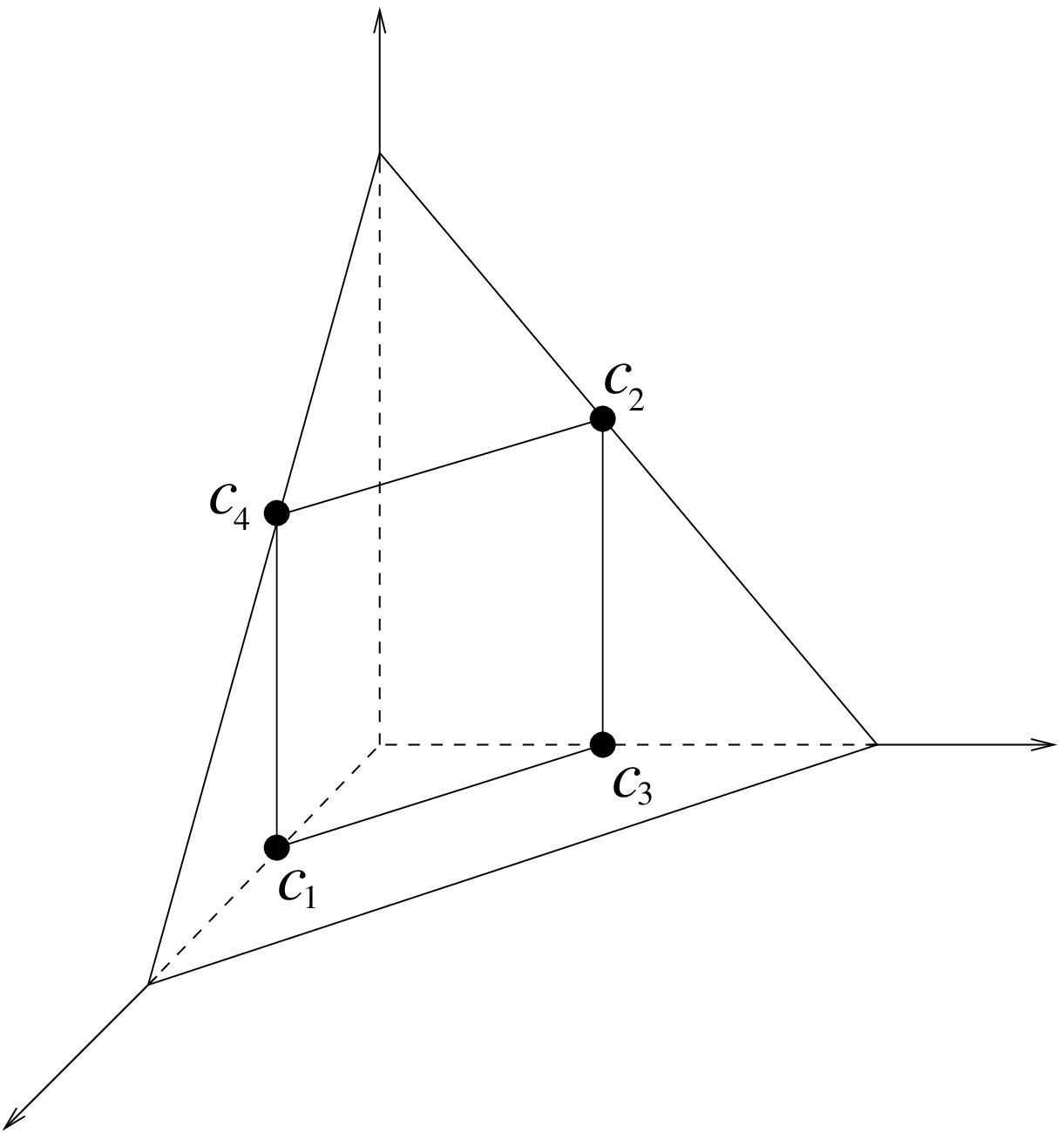, width=6cm} \epsfig{file=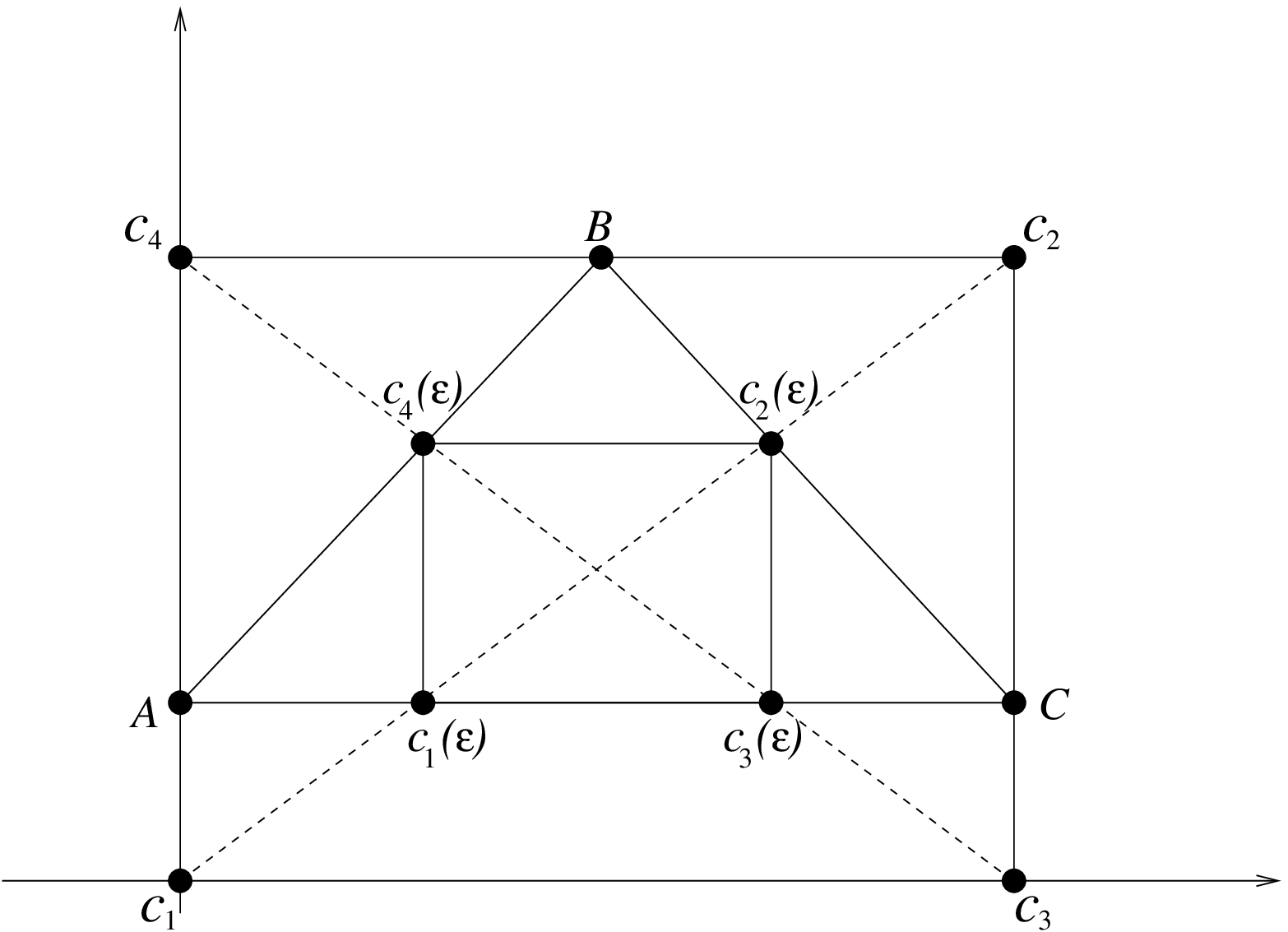, width=6cm}
\end{center}
\caption{The matrix $M_0$ in the simplex $\Delta^3$ (left) and the
points $c_1(\epsilon), \ldots, c_4(\epsilon)$ in the critical
configuration for $\epsilon = \sqrt{2}/2$
(right).}\label{tetratrian}
\end{figure}

The points $c_1(\epsilon),\ldots,c_4(\epsilon)$ are the vertices
of a rectangle $R_\epsilon$ which we can draw in the plane. As
$\epsilon >0$ increases, the four points move along the main
diagonals, as in Figure \ref{tetratrian} (right), and $R_\epsilon$ will
eventually be contained in the triangle $ABC$ where
\[
A=(0, {\sqrt{2}}/{2}-{1}/{2}) \qquad B= ({\sqrt{2}}/{4},{1}/{2})
\qquad C=({\sqrt{2}}/{2},{\sqrt{2}}/{2}-{1}/{2}).
\]

It is not hard to show that, for $\epsilon < {\sqrt{2}}/{2}$, we
have $\mathrm{rk}_+(M_\epsilon)=4$ while
$\mathrm{rk}_+(M_{{\sqrt{2}}/{2}})=3$. Hence, moving far enough
from $M_0$ the non-negative rank can decrease.

\noindent {\bf Non-convexity of $X^+_{4 \times 4,3}$.} In the $4
\times 4$ case, the properties of the non-negative rank imply that
the unique non-trivial case is the case of rank $3$. The matrices
in $X_{4 \times 4,3}$ can belong to $X^+_{4 \times 4,3}$ or to
$X^+_{4 \times 4,4}\setminus X^+_{4 \times 4,3}$. With the same
graphical approach as above, we can show that the set $X^+_{4
\times 4,3}$ is not convex (even if the ordinary rank is
constant). To do this, it is enough to consider the two matrices
$A_1=[c_4, c_2, c_3, f_1]$ and $A_2=[c_4, c_3, c_1, f_2]$ where
the columns $c_1, c_2, c_3, c_4, f_1, f_2$ are displayed in Figure
\ref{convex} in the same plane as in Figure \ref{tetratrian}
(right).
\begin{figure}
\begin{center}
\epsfig{file=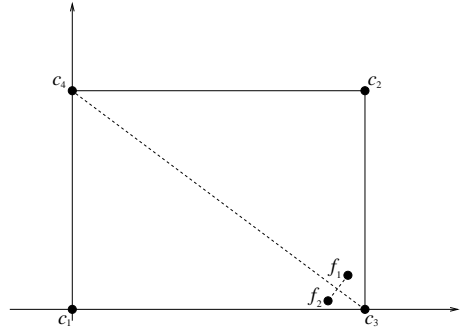, width=6cm}
\end{center}
\caption{The points $c_1, \ldots, c_4,f_1,f_2$ defining the
matrices $A_1$ and $A_2$.}\label{convex}
\end{figure}
It is immediate to see that both $A_1$ and $A_2$ have rank $3$ and
non-negative rank $3$, but the matrix $A=(A_1+A_2)/2$ has rank $3$
(its $4$ points are coplanar) but non-negative rank $4$. With the
same technique, one can also see that the set $X_{4 \times 4,3}
\setminus X^+_{4 \times 4,3}$ is not convex.
\[
B_1 = \begin{pmatrix}
1 & 0 & 1 & 0\\
1 & 0 & 0 & 1\\
0 & 1 & 1 & 0\\
0 & 1 & 0 & 1
\end{pmatrix} \qquad
B_2= \begin{pmatrix}
1 & 0 & 1 & 0\\
1 & 1 & 0 & 0\\
0 & 0 & 1 & 1\\
0 & 1 & 0 & 1
\end{pmatrix} \, .
\]
$B_1$ and $B_2$ have rank $3$ and non-negative rank $4$, as they
are obtained from $M_0$ possibly with permutation of columns, but
the matrix $B=(B_1+B_2)/2$ has non-negative rank $3$.

\section{Relations with the analysis of statistical mixture models}\label{statsection}

The results about the non-negative rank presented above have a
useful counterpart in Probability and Mathematical Statistics. In
particular the notion of nonnegative rank is useful in the study
of mixture of independence models for discrete distributions. We
now recall some basic definitions.

\noindent {\bf Distribution.} The distribution (or density) of a
random variable $X$ on a set of $n$ possible outcomes $\{1,
\ldots, n\}$ is a vector of $n$ non-negative numbers $(p_1,
\ldots, p_n)$ such that
\begin{equation*}
p_i \ge 0 \ \mbox{ for all } \ i \ \mbox{ and } \ \sum_i p_i = 1
\, ,
\end{equation*}
where $p_i= \mathbb{P}(X=i)$ is the probability that $X$ assumes
the value $i$.

\noindent {\bf Joint distribution.} If we consider a pair $(X,Y)$
of random variables on $\{1, \ldots, n\}$ and $\{1, \ldots, m \}$
respectively, the joint distribution of $X$ and $Y$ is a
probability matrix, i.e. a non-negative matrix $P=(p_{i,j})$ such
that
\begin{equation} \label{twowaytable}
p_{i,j} \ge 0 \ \mbox{ for all } \ i,j \ \mbox{ and } \ \sum_{i,j}
p_{i,j} = 1 \, ,
\end{equation}
where $p_{i,j}= \mathbb{P}(X=i, Y=j)$ is the probability that
$(X=i)$ and $(Y=j)$.

\noindent {\bf Probability models.} A matrix $P$ satisfying the constraints in Equation
\eqref{twowaytable} is also called a two-way table. The set
\begin{equation*}
\Delta = \left\{ P \in \mathbb{R}^{nm} \ : \ p_{i,j} \ge 0 \
\mbox{ for all } \ i,j \ \mbox{ and } \ \sum_{i,j} p_{i,j} = 1
\right\} \,
\end{equation*}
is the $n \times m$ (closed) standard simplex and each probability
distribution for a pair $(X,Y)$ belongs to $\Delta$. A probability
model ${\mathcal M}$ is a subset of $\Delta$. In many cases
${\mathcal M}$ is defined through a set of polynomial equations,
and in such case we call ${\mathcal M}$ an algebraic model.

\noindent {\bf The independence model.} For two-way tables, one
among the most simple models is the independence model. The
construction of the independence model is described for instance
in \cite{agresti:02}. Under independence of $X$ and $Y$ we have
\begin{equation*}
\mathbb{P}(X=i,Y=j) =  \mathbb{P}(X=i) \mathbb{P}(Y=j)
\end{equation*}
for all $i=1, \ldots, n$ and for all $j = 1, \ldots , m$, and
therefore $P$ is a rank one matrix, i.e., there exist vectors $r$
and $c$ such that $P = c (r)^t$. Thus, the independence model for
$n \times m$ tables is the set:
\begin{equation*}
{\mathcal M}_I = \left\{ P \ : \ \mathrm{rank}( P) = 1 \right\}
\cap \Delta \, .
\end{equation*}

{\em Remark.} It is a well known fact in Linear Algebra that a non-zero matrix
$P$ has rank $1$ if and only if all $2 \times 2$ minors of $P$
vanish. This shows that the independence model is an algebraic
model. Thus, an equivalent definition of the independence model is
as follows. The independence model is the set:
\begin{equation*}
{\mathcal M}_I = \left\{ P \ : p_{i,j}p_{k,h}-p_{i,h}p_{k,j} = 0  \mbox{ for all } \ 1 \le i < k \le n, 1 \le j < h \le m
\right\} \cap \Delta \, .
\end{equation*}
Notice that the model is defined through pure binomials and that
the set of all the $2 \times 2$ minors of a matrix are a system of
generator of a toric ideal. This is a general fact in the analysis
of algebraic statistical models and the models of this form are
called toric models. The reader can refer to
\cite{drton|sturmfels|sullivant:09} and \cite{rapallo:07} for
further details.

\noindent{\bf Mixture models.} The mixture of two independence
models is defined through the following procedure:
\begin{itemize}
\item Take two distributions $P_1, P_2 \in {\mathcal M}_I$;

\item Toss a (biased) coin and choose $P_1$ with probability
$\alpha$ and $P_2$ with probability $(1-\alpha)$.
\end{itemize}
It is clear that the resulting distribution is a convex
combination of $P_1$ and $P_2$, i.e., a matrix of the form $\alpha
P_1 + (1-\alpha) P_2$. This process can be generalized. We can
consider $k$ distributions $P_1, \ldots, P_k \in {\mathcal M}_I$
and define the mixture of $k$ independence models as follows. The
mixture of $k$ independence models is the set
\begin{equation} \label{kmixture}
{\mathcal M}_{kI} = \{ P \ : \ P = \alpha_1c_1 (r_1)^t + \ldots +
\alpha_k c_k (r_k)^t \} \, ,
\end{equation}
where the vectors $r_h$, the vectors $c_h$ and $\alpha=(\alpha_1,
\ldots, \alpha_k)$ are probability distributions, i.e. all the
components are non-negative and each vector has sum one. Some
results and examples about this type of statistical models are
presented in \cite{fienberg|hersh|rinaldo|zhou:10}.

Notice that in the decomposition in Equation \eqref{kmixture}, the
components must be non-negative, and therefore the model coincides
with the set of $n \times m$ matrices with non-negative rank at
most $k$ and with sum equal to one, i.e., ${\mathcal
M}_{kI}=X^+_{n \times m,k}$.

{\bf Maximum Likelihood Estimation.} The problem of MLE consists
in finding the global maxima of a suitable function, called
likelihood, $L: {\mathcal M}_{kI} \rightarrow {\mathbb R}$. This
problem is of great relevance and it has stimulated a lot of
research, both from the theoretical and the numerical point of
view. The interested reader can find a summary in
\cite{mclachlan|peel:00}. In the case of mixture models, MLE is
quite difficult mainly for two reasons: (a) while for a large
class of statistical models the likelihood is a concave function,
for mixture model this is not true; (b) the natural
parametrization of the model as in Eq. \eqref{kmixture} is
redundant, see \cite{carlini|rapallo:10c} for more on this. The
papers \cite{fienberg|hersh|rinaldo|zhou:10} and
\cite{govaert|nadif:10} present some {\it ad hoc} solutions for
mixture models with special interest in applications. Hence, a
precise investigation of the geometric structure of the
statistical model is essential to handle the maximization problem.

From the geometric investigation carried out in the previous
section, we get a negative result. It is not possible to
approximate a probability matrix $P$, with
$\mathrm{rk}_+(P)\neq\mathrm{rk}(P)$, using matrices with ordinary
rank and non-negative rank which coincide. This fact is summarized
in the following corollary.

\begin{corollary}\label{negativecorr}
Given $n,m$ and $2 < k < \min\{n,m\}$, the set ${\mathcal M}_{kI}$ is not
dense in $X_{n \times m,k}$.
\end{corollary}
\begin{proof}
This result is a straightforward application of Proposition
\ref{uppersemicontinuos}. Let $P$ be a non-negative matrix such
that $\mathrm{rk}(P)=k$ and $\mathrm{rk}_+(P)>k$. Then there is no
sequence of matrices $P_n$ whose limit is $P$ such that
$\mathrm{rk}(P_n)=\mathrm{rk}_+(P_n)=k$.
\end{proof}

We consider Corollary \ref{negativecorr} a negative result in the
following sense: to study MLE on mixture models one must consider
matrices with non-negative rank different form the ordinary rank.
In particular, it is necessary to investigate the (not clear and
not trivial) geometry of the set  $X^+_{n \times m,k}$ and of its
boundary. This study is necessary in order to be able to exploit
optimization techniques and to avoid redundant variables.

\section*{Acknowledgments}
The authors thank the two anonymous referees and the editor for their suggestions and comments which led to a better paper.

\end{document}